\documentclass[12pt]{article}
\usepackage{tikz}
\usetikzlibrary{arrows,snakes,backgrounds}

\usepackage[centertags]{amsmath}
\usepackage{amsfonts}
\usepackage{amssymb}
\usepackage{amsthm}
\usepackage{amsmath,here}
\usepackage{color}
\usepackage{url}
\usepackage[vcentermath,enableskew,noautoscale,stdtext]{youngtab}
\usepackage[lined,algonl,boxed]{algorithm2e}

\theoremstyle{plain}
\newtheorem{thm}{Theorem}
\newtheorem{cor}[thm]{Corollary}
\newtheorem{lem}[thm]{Lemma}
\newtheorem{prop}[thm]{Proposition}
\theoremstyle{definition}
\newtheorem{defn}{Definition}
\theoremstyle{definition}

\title{Linear time algorithm for computing the rank of divisors on cactus graphs}
\author{
Phan Thi Ha Duong\footnote{Institute of Mathematics, Vietnam Academy of Science and Technology, Email: \textsf{phanhaduong@math.ac.vn}} 
}
\date{}
\begin{document}

\maketitle
\begin{abstract}
Rank of divisor on graph was introduced in 2007 and it quickly attracts many attentions. Recently in 2015, the problem of computing this quantity was proved to be NP-hard. In this paper, we describe a linear time algorithm for this problem limited on cactus graphs. 

\end{abstract}
\vspace{0.5cm}
%


\section{Introduction}
The notion of rank of divisor on graph was introduced by Baker and Norine in a paper on Jacobi-Abel theory on graph \cite{BN07}, in which the authors stated the link between this notion with similar notion on Riemann surface. Moreover, the authors have developed a theorem for divisor on graph analogue to the classical Riemann-Rich theorem. Since then, many works have studied for computing the rank of divisor on graph (see for example \cite{CLM15}). The most important result should be the new theorem on the NP-hardness complexity of  rank of divisor problem on general graph \cite{KT15}. The proof of this result was based on the proof of NP-hardness of minimum recurrent configuration problem of Chip Firing Game on directed graphs studied by Perrot and Pham \cite{PP15}. On the other hand, the rank of divisor problem can be studied in special classes of graphs. In \cite{CL14}, the author proposed a linear time algorithm for this problem on complete graph. The idee of this algorithm is based on Dyck words and parking function, notions very closed to Chip Firing Game, a very well-known combinatorial model \cite{BLS91, GLMMP01,  PP13}

In this paper we investigate this problem in the case of cactus graph. This class was introduced in 1950's year \cite{FG53}, and can be used for representing model on different research domains, for example electrical circuits \cite{Tetsuo91,Yu83} or comparative genomics \cite{PDESMSH}. Several NP-hardness problem on general graphs can be solved in polynomial time on cactus graphs \cite{Das10, Hobbs79, ZZ03}.

Our main idea is to contract a graph by eliminating edge and cycle, and deduce the rank of divisor on initial graph from that of contracted graph. For a general graph, such an elimination does not always exist, but it is the case for cactus graph. We show that a block (edge or cycle) elimination scheme can be found in linear time for cactus graph, furthermore from this scheme, we construct an algorithm in linear time for computing the rank of divisors. 

In Section 2, we will present the key features of the theory of Riemann Roch on graph. Then we discuss about the rank of divisors on trees and cycles. We propose the contraction operator on graph by eliminating edge or cycle, and take in evidence the relation between the rank of divisors on initial graph and that on its contraction.

Section 3 focuses on cactus graphs, on the construction of a block elimination scheme, and from there a linear time algorithm for computing the rank of divisors.

\section{Divisors on graphs and Riemann-Roch like theorem}

Let $G$ be a multiple undirected graph that has no loop. We always denote by $V(G)$ the vertex set of $G$ and by $n$ its cardinality, by $E(G)$ the edge set of $G$ and by $m$ its cardinality.  For each vertex $v \in V(G)$, we write $deg(v)$ the degree of $v$, and for every vertices $u, v \in V(G)$, we write $e(u,v)$ the number of edges between $u$ and $v$. The \emph{genus} $g$ of $G$ is the quantity $g=m-n+1$. For a subset $U$ of $V$, we denote by $G(U)$ the subgraph of $G$, induced by $U$.

The group of divisors of $G$, $Div(G)$ is the free abelian group on $V(G)$. A divisor $f \in Div(G)$ can be considered as a function $f: V \rightarrow \mathbb{Z}$, or as a vector $f \in \mathbb{Z}^{V(G)}$, where the coordinates are indexed by the vertices of $G$. The degree of $f$ is defined by $deg(f) =\sum_{v \in V(G)}f(v)$. The index vector $\epsilon_v$ is defined by a vector of entries $0$ except $\epsilon_v(v) = 1$.


The Laplacian matrix $(\Delta_G)_{n\times n}$ of graph $G$,  where the coordinates are indexed by $V(G) \times V(G)$, is defined by:
$$\Delta_G(u,v) = \left\{\begin{array}{ll}
deg(u) & \mbox{ if } u=v,\\
-e(u,v)&  \mbox{ if } u \neq v.
\end{array}\right.
$$

We write $\Delta_G(v)$ the vector indexed by vertex $v$ of the matrix.

A divisor $f \in Div(G)$ is called \emph{effective} if $f(v) \geq 0$ for all $v \in V$.

The \emph{linear equivalence} is a  relation on $Div(G)$ defined by: $f \sim g$ if there exists $x \in \mathbb{Z}^{V(G)}$ such that $g = f + x \Delta_G $. If $f$ is linear equivalent with an effective divisor $g$, we say $f$ is $L$\emph{-effective}.

We give here the definition of the rank of divisor which was introduced by Baker and Norine \cite{BN07}.
\begin{defn}
	For a divisor $f \in Div(G)$, the \emph{rank} of $f$ is 
	\begin{itemize}
		\item $-1$  if $f$ is not effective,
		\item the largest integer $r$  such that for any effective configuration $\lambda$ of degree $r$ the divisor $f - \lambda $  is L-effective.
	\end{itemize}
\end{defn}
It is useful to state a straightforward property of the rank. 
\begin{lem}
	Let $f$ and $f'$ be two divisors of degree non negative on $G$. Then $\rho(f+f') \leq \rho(f) + deg(f')$. In particular, for every $v \in V(G)$, we have $\rho(f) - 1 \leq \rho(f-\epsilon_v) \leq \rho(f)$.  
\end{lem}

In their first paper on the rank of divisor, Baker and Norine have proved the following theorem which is analogue to the Riemann Roch theorem on Riemann surface.

\begin{thm}
Let $G$ be a graph with $n$ vertices and $m$ edges. Let $\kappa$ be the divisor such that $\kappa(v) = d(v)  - 2$ for all $v \in V(G)$, so that
$deg(\kappa) = 2(m - n)$. Then any divisor $f$ satisfies:
$$\rho(f) - \rho(\kappa - f)  = deg(f) -g+1,$$
where $g$ being the genus of $G$.
\end{thm}
Let us remark that for any divisor $f$ such that $deg(f)<0$ then $f$ is not effective, and $\rho(f) = -1$. Moreover $deg(\kappa -f) = deg(\kappa) - deg (f) = 2(m-n) - deg(f)$. Then if $deg(f) > 2(m-n)$, we have $deg(\kappa-f)<0$ and $\rho(f)=-1$, this implies that $\rho(f) = deg(f)-g$.
\subsection{Rank on trees and cycles}
We now investigate to some elementary cases of graphs. 
\subsubsection{Tree}
A tree is an acyclic connected graph. In a tree, we have $m=n-1$, and $2(m-n)=-2$. So for every $f$ of degree non negative, we have $deg(f) > deg(\kappa)$ which implies that $\rho(f) =deg(f)$.
\subsubsection{Cycle}
A cycle is a connected graph where every vertex are of degree 2. In a cycle of $n$ vertices $C_n=\{v_1,\ldots, v_n\}$, we have $m=n$ and $g=1$. So for every $f$ of degree positive, we have $deg(f) > deg(\kappa)$ which implies that $\rho(f) =deg(f)-1$.

In the case $deg(f)=0$, $\rho(f)=0$ if $f \sim 0$ (that means $f$ is $L$-effective), otherwise $\rho(f)=-1$. We call {\em good divisor} a divisor of degree $0$ and L-effective, and {\em bad divisor} a divisor of degree $0$ and not L-effective. For cycle $C_n$, we can write a divisor $f$ on $C_n$ as a vector $f=(f_1, f_2, \ldots, f_n)$. 

\begin{prop}
	Let $f=(f_1, f_2, \ldots, f_n)$ be a divisor on cycle $C_n$, then the rank of $f$ is computed as follows.
	$$\rho(f) = \left\{\begin{array}{llll}
	-1 & \mbox{ if } deg(f) \leq -1,\\
	-1 &  \mbox{ if } deg(f) =0 \mbox{ and } f \mbox { is bad},\\
	0 &  \mbox{ if } deg(f) =0 \mbox{ and } f \mbox{ is good },\\
	deg(f)-1 & \mbox{ if } deg(f) \geq 1.\\
	\end{array}\right.
	$$
\end{prop}

Now, we analyze the characterization of good divisors on cycles.


Let $f=(f_1, f_2, \ldots, f_n)$ be a divisor on cycle $C_n$. We have $f \sim 0$ if and only if there exists $x=(x_1, x_2, \ldots, x_n) \in \mathbb{Z}_n$ such that $f - x \Delta_{C_n} = 0$. Because $\sum_{i=1}^{i=n} \Delta_{C_n}(v_i) = 0$ then we have 
$$f \sim 0 \Leftrightarrow \exists x=(x_1, x_2, \ldots, x_{n-1}, 0) \in \mathbb{Z}_{n-1} \times \{0\}:   f - x \Delta_{C_n} = 0.$$
$$ \Leftrightarrow  \exists x: (f_1, f_2, \ldots, f_n)  = (x_1, x_2, \ldots x_{n-1},0)  
\left( \begin{array}{ccccccc} 
2 & -1 &0 &\ldots &0&0 & -1  \\
-1& 2 & -1 & \ldots &0&0 & 0\\
...\\
0 & 0 & 0 &\ldots &-1&2&-1 \\
-1 & 0 & 0 &\ldots &0&-1&2 \\
\end{array} \right)
 $$
$$ \Leftrightarrow  \exists x: 
\left\{\begin{array}{lll}
            f_1 & = & 2x_1 - x_2\\
            f_2 & = & - x_1 + 2x_2 - x_3 \\
            f_3 & = & - x_2 + 2x_3 - x_4 \\
            \ldots\\
	     f_{n-1} & = & - x_{n-2} + 2x_{n-1}\\
	     f_n & = & 2x_{n-1} - x_1\\ 	
            \end{array}\right.
$$
$$ \Leftrightarrow  \exists x: 
\left\{\begin{array}{lllll}
            x_2 & = & 2x_1 - f_1\\
            x_3 & = & 3 x_1 -(2f_1+f_2) \\
            x_4 & = & 4 x_1 -(3f_1+2f_2+f_3) \\
            \ldots\\
	     x_{n-1} & = & (n-1)x_1- ((n-2)f_1+\ldots +2f_{n-3} + f_{n-2})\\
	     0 = x_n & = & nx_1 -((n-1)f_1 + (n-2) f_2 + \ldots + 2f_{n-2} +f_{n-1})\\ 	
            \end{array}\right.
$$
$$ \Leftrightarrow  ((n-1)f_1 + (n-2) f_2 + \ldots + 2f_{n-2} +f_{n-1}) \equiv 0 \mod{n}$$
$$ \Leftrightarrow  (f_1 + 2 f_2 + \ldots + (n-2)f_{n-2} + (n-1)f_{n-1}) \equiv 0 \mod{n}.$$


So we have the following result.
\begin{prop}
	Let $f=(f_1, f_2, \ldots, f_n)$ be a divisor of degree $0$ on the cycle $C_n$, then $f$ is good if and only if  
$$f_1 + 2 f_2 + \ldots + (n-2)f_{n-2} + (n-1)f_{n-1} \equiv 0 \mod{n}.$$
\label{Prop: Good}
\end{prop}

\subsection{Operators and rank of divisors}

The two simple cases of trees and cycles give us the idea to decompose a graph to smaller graphs in a way that the rank of the initial graph can be deduced from that of smaller graphs.

To this purpose, we introduce two operators on graph and on its divisors.

\begin{defn}
	Let $G$ be a connected graph. A vertex $v$ of $G$ is called a \emph{cut vertex} if removing $v$ from $G$ disconnects $G$. Moreover, if one can decompose $V(G)=V_1 \cup U$ such that $V_1 \cap U = \{v\}$ and that the induced graphs $G_1=G(V_1)$ and $H=G(U)$ are connected, we say $v$ decomposes $G$ into $G_1$ and $H$. We will denote by $G /H$, and call {\em the contraction} of $G$ by $H$ at vertex $v$,  the subgraph $G_1$.
	
	Furthermore, if $H$ is a block (maximal subgraph without a cut-vertex) we say $v$ a {\em block cut vertex} and $H$ a {\em free block} of $G$. 
\end{defn}	

\begin{defn} Let $G$ be a graph, and let $v$ be a cut vertex which decomposes $G$ to $G_1=G(V_1)$ and $H=G(U)$. Let $f$ be a divisor on $G$, we define {\em contraction} of $f$ by $H$, and denote by $f_{G / H}$, the following divisor on $G/H$.
	
	$$f_{G/H}(u) = \left\{\begin{array}{lll}
	f(u) & \mbox{ if } u \in V_1 \backslash\{v\}, \\
	\sum_{u \in H} f(u) &  \mbox{ if } u=v. 
	\end{array}\right.
	$$
	
	We define {\em zero} of $f$ on $H$, and denote by $f_{N(H)}$, the following divisor on $H$.
	
	$$f_{N(H)}(u) = \left\{\begin{array}{lll}
	f(u) & \mbox{ if } u \in H \backslash\{v\}, \\
	- \sum_{u \in H \backslash\{v\}} f(u) &  \mbox{ if } u=v. 
	\end{array}\right.
	$$
\end{defn}
One has directly the relation between a divisor and its contraction and zero.
$$\left\{\begin{array}{lll}
f_G = f_{G/H} + f_{N(H)},\\
deg(f_{G/H}) = deg(f_G),\\
deg(f_{N(H)})=0. 
\end{array}\right.
$$
Generally, let $U$ be a subset of $V(G)$ and let $H=G(U)$, we can consider a divisor on $H$ as a divisor on $G$ by giving value $0$ to all vertices in $V(G) \backslash U$. Similarly, we consider the matrix indexed by vertices of $H$ as a matrix indexed by vertices of $G$ by giving value $0$ to all entries indexed by vertices in $V(G) \backslash U$.

It is easy to check that $\Delta_{G
/H} + \Delta_{H} = \Delta_G$.

Nevertheless, the rank of a divisor on $G$ and on $H$ are not the same, that means if $f$ is a divisor on $H$ then $f$ can be seen as a divisor on $G$ but $\rho_G(f) \neq \rho_H(f)$. 

Now we show that the rank of a divisor can be computed from that of its contraction. 

\begin{prop}
	Let $G$ be a graph and let $v$ be a cut vertex which decomposes $G$ to $H$ and $G_1$. If $H$ is a tree then for all divisor $f$ on $G$, we have $\rho(f) = \rho(f_{G/H})$.
\end{prop}
\begin{proof}
	Let $r$ be the rank of $\rho(f_{G/H})$, we will prove that $\rho(f) =r$.
	Let $\lambda$ be a divisor on $G$, we have:
	$$f-\lambda = (f-\lambda)_{G/H} + (f-\lambda)_{H}.$$ 
	Because $(f-\lambda)_{H}$ is of degree $0$ on a tree then it is  $L_G$-effective.
	This implies that $f-\lambda$ is $L_G$-effective if and only if  $(f-\lambda)_{G/H}$ is $L_G$-effective, so $\rho(f) = \rho(f_{G/H})$.
\end{proof}

From the above result, we observe that one can contract a graph by its tree and the rank of a divisor does not change after this contraction. After that, the result graph has no vertex of degree 1. It turns out that we need to focus only connected graphs whole every vertices are of degree at least two.

The situation will be more complicated for contraction by a cycle because a divisor of degree $0$ on a cycle can be good or bad. 

\begin{prop}
	Let $G$ be a graph and let $v$ be a cut vertex which decomposes $G$ into $H$ and $G_1$ where $H$ is a cycle. Let $f$ be a divisor on $G$. If $f_{N(H)}$ is bad then $\rho(f) = \rho(f_{G/H}- \epsilon_v)$.
	\label{Prop: bad cycle}
\end{prop}
\begin{proof}
	Put $r= \rho(f_{G/H})$. 
	Let us consider $\rho(f_{G/H} -  \epsilon_v)$ which can be $r$ or $r-1$.
	
	If  $\rho(f_{G/H} -  \epsilon_v) = r$. Let consider any divisor $\lambda$ of degree $r$. We have $f-\lambda = (f_{G/H} - \epsilon_v \lambda_{G/H}) +  (f_{H}+ \epsilon_v-\lambda_{H})$, but  $f_{G/H} - \epsilon_v- \lambda_{G/H}$ is $L_G$-effective because $\rho(f_{G/H} -\epsilon_v) = r$ and $deg(\lambda_{G/H} = r$, and  $f_{H}+ \epsilon_v-\lambda_{H}$ is a divisor of degree $1$ on cycle $H$ then $L_G$-effective; we have then $f-\lambda$ is $L_G$-effective, which give the rank $r$ for $f$.
	
	Now if $\rho(f_{G/H} -  \epsilon_v) =r-1$ then there exists $\lambda$ on $G/H$ of degree $r$ such that $f_{G/H} - \lambda -  \epsilon_v$ is not $L_G$-effective. Consider $\lambda$ as a divisor on $G$, then $f_{H}-\lambda_{H}= f_{H}$ which is bad. Then to make the part on $H$ positive, we must take at least $1$ unit from $V_2$. That mean the part  $f_{G/H} - \lambda$ on $G/H$ must give at least $\epsilon_v$ to the part on $H$. But we know that  $(f_{G/H} - \lambda) - \epsilon_v$ is not $L_G$-effective, then it is impossible. 
	
	We can conclude that there exists a divisor $\lambda$ of degree $r$ such that $f - \lambda$ is not $L_G$-effective, then $\rho(f) = r-1$. 
\end{proof}

\begin{prop}
	Let $G$ be a graph and let $v$ be a cut vertex which decomposes $G$ into $H$ and $G_1$ where $H$ is a cycle. Let $f$ be a divisor on $G$. If $f_{N(H)}$ is good then we can compute $\rho(f)$ as follows.
	$$\rho(f) = \left\{\begin{array}{ll}
	r & \mbox{ if } \rho(f_{G/H} - 2 \epsilon_v) \geq r-1,\\
	r-1 &  \mbox{ if } \rho(f_{G/H} - 2 \epsilon_v) = r-2,
	\end{array}\right.
	$$
	where $r=\rho(f_{G/H})$.
	\label{Prop: good cycle}
\end{prop}
\begin{proof}
	Put $r=	\rho(f_{G/H})$. Let us consider $\rho(f_{G/H} - 2 \epsilon_v)$, this value can be $r$, $r-1$ or $r-2$. 
	
	If $\rho(f_{G/H} - 2 \epsilon_v) \geq r-1$ then for all divisor $\theta$ on $G/H$ of degree $r-1$, one has $f_{G/H} - 2 \epsilon_v - \theta$ is $L_G$- effective. That means there exists $\alpha$ such that 
	$$()f_{G/H} - 2 \epsilon_v - \theta) - \sum_{u \in V_2 \backslash \{v\}} \alpha_u \Delta_u \geq 0.$$
	
	Now let consider any divisor $\lambda \geq 0$ of degree $r$, we will prove that $f-\lambda$ is $L_G$-effective. 
	
	Let consider $\lambda_{G/H}$, if  $\lambda_{G/H}(v)=0$, that means $\sum_{u \in V_1} \lambda(u) = 0$ which implies that for all $u \in V_1$, $\lambda(u)=0$. Then $f_{H}-\lambda_{H}= f_{H}$ which is good and then $L_G$-effective. 
	On the other hand $f-\lambda = f_{G/H} - \lambda_{G/H} +  f_{H}-\lambda_{H}$, but  $f_{G/H} - \lambda_{G/H}$ is $L_G$-effective by hypothesis of the rank $f_{G/H}$ of  (note that $deg(\lambda_{G/H} = r$), then $f-\lambda$ is $L_G$-effective.
	
	Now, if $\lambda_{G/H}(v) \geq 1$. Put $\theta = \lambda_{G/H}(v) - \epsilon_v$. Then $deg(\theta)=r-1$, and we have $f_{G/H} - 2 \epsilon_v - \theta$ is $L_G$- effective. That means $f_{G/H}- \epsilon_v - \lambda_{G/H}  = f_{G/H} -  \epsilon_v - (\theta + \epsilon_v)$  is $L_G$- effective. 
	On the other hand $f-\lambda = (f_{G/H} - \epsilon_v -\lambda_{G/H} +  (f_{H}+ \epsilon_v -\lambda_{H})$, but $f_{H}+ \epsilon_v -\lambda_{H}$ is of degree $1$ on a cycle $H$ then it is $L_G$-effective, then we have $f-\lambda$ is $L_G$-effective. 
	
	We prove also that if $\rho(f_{G/H} - 2 \epsilon_v) = r-2$ then $\rho(f) = r-1$. 
	
	In fact, because $\rho(f_{G/H} - 2 \epsilon_v) = r-2$ then there exists a divisor $\theta$ on $G/H$ such that $\theta \geq 0$, $deg(\theta)=r-1$ and $f_{G/H} - 2 \epsilon_v - \theta$ is not $L_G$-effective. Which is equivalent to $(f_{G/H} - (\theta+\epsilon_v)) - \epsilon_v$ is not $L_G$-effective.
	
	Now define the divisor $\lambda$ on $G$ by $\lambda = \theta +\epsilon_v - \epsilon_v + \epsilon_w$  with $w \in V_1 \backslash \{v\}$, then $f-\lambda =(f_{G/H} - (\theta+\epsilon_v)) +(f_{H} +\epsilon_v-\epsilon_w)$. Let us consider the divisor $f_{H} +\epsilon_v-\epsilon_w$, this is a divisor of degree $0$ on $H$ and is not good because $f$ is good. So if we want to change this divisor to a positive divisor, we must take at least $1$ unit from $V_2$. That mean the part  $f_{G/H} - (\theta+\epsilon_v)$ on $G/H$ must give at least $\epsilon_v$ to the part on $H$. But we know that  $(f_{G/H} - (\theta+\epsilon_v)) - \epsilon_v$ is not $L_G$-effective, then it is impossible. 
	
	We can conlude that there exists a divisor $\lambda$ of degree $r$ such that $f - \lambda$ is not $L_G$-effective, then $\rho(f) = r-1$. 
\end{proof}

The two above Propositions give us an idea to compute the rank of a divisor by an elimination scheme of trees and  cycles (if it exists). Unfortunately, for a general graph, we can not reduced a graph to a simple vertex by a sequence of contraction of tree and cycles. Moreover, if we can do it, we must consider two cases for the elimination of a good cycle. The last one can make the algorithm an exponential number of computations. 

But, for the cactus graph, one can overcome these two difficulties. We will show the existence of elimination scheme of tree and cycles on cactus; and we will prove that there is only one case for Proposition \ref{Prop: good cycle}.

\section{Rank of divisor on cactus graph}
\subsection{Cactus graph and block elimination scheme}
\begin{defn}
	A cactus graph (sometimes called a cactus tree) is a connected graph in which any two simple cycles have at most one vertex in common. Equivalently, every edge in such a graph belongs to at most one simple cycle. Equivalently, every block (maximal subgraph without a cut-vertex) is an edge or a cycle.
\end{defn} 
It is easy to see that in a cactus every cycle is simple and the number of simple cycles of a cactus $G$ is equal to its genus  $g=m-n+1$.

In our study, we are interest on special cut vertex which will give us an elimination scheme of cactus.

%

\begin{defn}
	We say that the graph $G$ have a {\em block elimination scheme} if one can construct a sequence of graphs $G_0, G_1, \ldots G_k$, $k \geq 0$, such that $G_0=G$, $G_k$ has only one vertex, denoted by $r$ or $v_{k+1}$, and for all $1 \leq i \leq k$, $G_i$ is obtained from $G_{i-1}$ by contraction a block $B_i$ of $G_{i-1}$ at vertex $v_i$. We denote this scheme by $E=(G_i,v_i,B_i. 1\leq i\leq k)$.  
	
Moreover, on the vertex set $\{v_1, \ldots, v_k, r=v_{k+1}\}$, we define the BES tree the tree rooted at $r$ and for every vertex $v_i$, $1\leq i \leq k$, the parent of $v_i$ is the vertex $v_j$ (with smallest index $j$) such that $v_i$ belongs to block $B_j$. 
\end{defn}

We can remark that if a graph accept a block elimination scheme, then by reverse the order of the sequence of block, one can obtain another block elimination scheme, this implies that this graph has at least two schemes, or there are at least two choices for the first block of a scheme. 

\begin{prop}
	A connected graph $G$ has a block elimination scheme if and only if $G$ is a cactus graph.
\end{prop}

\begin{proof}
	a) Let $G$ be a cactus graph. We show that $G$ has a block vertex $v$ (and its corresponding block $B$), and after taking $G_1$ obtained from $G$ by contraction $B$ at $v$, we prove that $G_1$ is also a cactus graph. We can continue this process to construct a block elimination scheme. 
	
	Suppose that $G$ has no block cut vertex. First, it implies that every vertex of $G$ has degree at least 2. Second, for every cycle $C$ of $G$, there is at least two vertices of $C$ of degree greater than 2. Indeed, if there is a cycle $C$ in which only one vertex $v$ is of degree greater than 2, then $v$ is a block cut vertex and $C$ is its corresponding block. 
	
	Let us consider the following path. Beginning from a vertex $u_1$ of degree greater than 2 of a cycle $C_1$, go to the second $v_1$ of degree greater than 2 by a path connecting $u_1$ to $v_1$ inside $C_1$. Then from $v_1$ go out of $C_1$ (this is possible because $deg(v_1) >2$). Continue the path, each time this path go into a new cycle $C_i$ by a vertex of degree greater than 2 $u_i$, it will go inside $C_i$ to a second vertex $v_i$ of degree greater than 2, then go out. This process will stop when either i) it returns to a vertex $w$ in $p$ and there is no cycle appear in $p$ more than twice, ii) it returns to a cycle $C_i$ by a vertex $w_i$ (which may different from $u_i$ and $v_i$). 
	
	Now let us consider the case i): the path $p$ is a cycle which is different from all cycles having intersection with $p$. Nevertheless, $p$ has two common vertices with $C_1$. This fact contradicts the property of cactus graph of $G$.   
	
	Suppose that we have the case ii). If $w_i$ is equal to $u_i$ or $v_i$, than $p$ contains a cycle which has two commun vertices with $C_i$. If $w_i$ is different from $u_i$ ans $v_i$. Then lest us consider the path $q$:  taking the sub path of $p$ from $u_i$ to $w_i$ and adding a path from $w_i$ tp $u_i$ inside $C_i$ (which does not contains $v_i$). This path $q$ is a cycle, which is different $C_i$ and which has two common vertices with $C_i$. We have then a contradiction. 
	
	After all, if $G$ has a block cut vertex $v$ (with block $C$), then the construction $B$ from $G$ at vertex $v$ is clearly a cactus graph.
	
	So we can conclude that a cactus graph $G$ has a block elimination scheme.  
	
	b) Now, if $G$ is not a cactus graph, we prove that $G$ has no block elimination scheme. If $G$ is not a cactus graph then there exists an edge $(u,v)$ which belongs to two simple cycles $C_1$ and $C_2$. The first time $(u,v)$ is contracted by an contraction operation, if $C_1$ is contracted then $C_2$ remains, but if $C_1$ is contracted then $u$ and $v$ contract to the same vertex while remaining $C_1$ means $u$ and $v$ remain different. So $G$ can not have a block elimination scheme.	 
\end{proof}
The recognition problem can be solved in linear time \cite{ZZ03} by using a depth first search. We use a similar idee to prove the following result. 
\begin{lem}
A block elimination scheme of  a cactus graph $G$ can be found in linear time.
\end{lem}
\begin{proof}
We will construct a tree and prove that this corresponds to a BES tree.

Let $r$ be any vertex of $G$. We call a depth first search (DFS) procedure for $G$ from $v$. This DFS give us a tree from which we will construct to obtain our tree $T$. In this DFS procedure, each cycle $C$ has an unique vertex $v$ who appears firstly in the DFS, and we represent this cycle by node $v$. Similarly, each edge $e$ which does not belong to any cycle has an extremity $u$ firstly appear in the DFS, and we represent this edge by node $u$. A node $x$ is a child of a node $y$ if either vertex $x$ is a sun of the vertex $y$ in the DFS tree and the edge $(x,y)$ does not belong to any cycle or if $x$ belongs to the cycle having $y$ as representation. 

After this contraction of the DFS tree, we obtain a tree $T$ where each node represent a block (a cycle or an edge) of cactus $G$. Moreover each leave $v$ of $T$ represent a block $B$ having $v$ as its block cut vertex in $G$. We can then construct a block elimination scheme of $G$ by contraction consecutively  block at leave by leave. 

Finally, a DFS procedure takes $O(m)$ time, then this construction takes $O(n)=O(m)$ times as claim. 
\end{proof}
                                                                                                                                                                                \subsection{Rank of divisors on cactus graphs}                                                                                                   
As we remark above on Proposition \ref{Prop: good cycle}, for general graph, there two cases for computing the rank of a divisor from its contraction; the situation will be simpler for cactus. For this purpose, we first prove the following result.
\begin{lem}
	Let $G$ be a cactus graph and let $v$ be a vertex of $G$. Let $f$ be a divisor on $G$, then $\rho(f_{G} - 2 \epsilon_v) < \rho(f_{G})$. 
\end{lem} 
\begin{proof}
	We prove by recurrence on $g(G)$.
	
	In the case $G$ of genus 0. then $G$ is a tree, we have  $\rho(f_{G} - 2 \epsilon_v) = deg(f_{G} - 2 \epsilon_v) < deg(f_{G}) = \rho(f_{G})$.

	Suppose that the statement of Proposition is correct for all cactus all genus smaller $k \geq 2$, we will prove it is correct for cactus of genus $k$. Let us consider a block cut vertex $v_1$ which decompose $G$ into $H$ and $G_1$ and such that $v \not in H$ (such a block exists always by a remark after the definition of block elimination scheme).
	
	Now, consider graph $G_1$, vertex $v_1$ and divisor $f_{G_1}$	, one has $\rho(f_{G_1} - 2 \epsilon_{v_1}) < \rho(f_{G_1})$ by hypothesis of recurrence. 
	
	If $H$ is a tree then $\rho(f_{G} - 2 \epsilon_v) = \rho(f_{G_1} - 2 \epsilon_v) < \rho(f_{G_1}) = \rho(f_{G})$.
	
	If $H$ is a bad cycle then  $\rho(f_{G} - 2 \epsilon_v) = \rho(f_{G_1} - 2 \epsilon_v - \epsilon_{v_1}) = \rho((f_{G_1} - \epsilon_{v_1}) - 2\epsilon_v) < \rho(f_{G_1} - \epsilon_{v_1}) = \rho(f_{G})$.
	
	If $H$ is a good cycle, and because $\rho(f_{G_1} - 2 \epsilon_{v_1}) < \rho(f_{G_1})$, then 
	$\rho(f_{G} - 2 \epsilon_v) = \rho((f_{G_1} - 2 \epsilon_v) - 2\epsilon_{v_1})+1 = \rho((f_{G_1} - 2\epsilon_{v_1}) - 2\epsilon_v)+1 < \rho(f_{G_1} - 2\epsilon_{v_1}) +1 = \rho(f_{G})$.
	
	So in anycase, we have always 	$\rho(f_{G} - 2 \epsilon_v) < \rho(f_{G})$ for $G$ of genus $k$.
	
	Which complete the recurrence argument. 
\end{proof}
	
From Proposition \ref{Prop: Good} and Proposition \ref{Prop: good cycle}, we have directly the following result.	
\begin{cor}
		Let $G$ be a cactus graph and let $v$ be vertex which decomposes $G$ into two graphs $H$ and $G_1$ where $H$ is a cycle. Let $f$ be a divisor on $G$ such that $f_{N(H)}$ is good. Then $\rho(f_G) =\rho(f_{G/H} - 2 \epsilon_v) +1$.
\end{cor}
	
We can now prove our main result.
\begin{thm}
	Let $G$ be a cactus, and let $E=({G_i}, {v_i}, {H_i}, 1 \leq i \leq k)$ be a block elimination scheme of $G$. Then we can compute the rank of any divisor $f$ on $G$ by the following recursive algorithm in linear time.
	
	For all $0 \leq i \leq k$:
	$$\rho(f_{G_i}) = \left\{\begin{array}{lll}
	\rho(f_{G_{i+1}}) & \mbox{ if } H_{i+1} \mbox{ is an edge, } \\
	\rho(f_{G_{i+1}} - \epsilon_{v_{i+1}}) & \mbox{ if } H_{i+1} \mbox{ is an cycle and } f_{N(H_i)} \mbox{ is bad,}\\
	 \rho(f_{G_{i+1}} - 2\epsilon_{v_{i+1}}) + 1 & \mbox{ if } H_{i+1} \mbox{ is an cycle and } f_{N(H_i)} \mbox{ is good.}\\
	\end{array}\right.
	$$
\end{thm}

\begin{proof}
	The correctness of this algorithm can be deduced from the above propositions ans corollary. We will now prove the complexity.
	
	Given a block elimination scheme, in the step $i$, one must calculate $f_{G_i}$. Firstly, it is in constant time to check if $H_i$ is an edge or a cycle. Then if $H_{i+1}$ is a cycle, it is $O(|H_{i+1}|)$ time to check if $f_{N(H_i)}$ is good or bad. In each case, one must calculate    $\rho(f_{G_{i+1}} - \epsilon_{v_{i+1}})$ or $\rho(f_{G_{i+1}} - 2\epsilon_{v_{i+1}})$, which is the recursive procedure on a new graph $G_{i+1}$ with the size smaller than that of $G_i$ a value of $O(|H_{i+1}|)$. 
	
	Totally, the algorithm takes a time of $O(|G|) = O(n)$.   
	
\end{proof}

\bibliography{BIBLIO_Cactus}

\begin{thebibliography}{10}

\bibitem{BN07}
Matthew Baker and Serguei Norine.
\newblock Riemann-{R}och and {A}bel-{J}acobi theory on a finite graph.
\newblock {\em Adv. Math.}, 215(2):766--788, 2007.

\bibitem{BLS91}
A.~Bjorner, L.~Lov\'asz, and W.~Shor.
\newblock Chip-firing games on graphes.
\newblock {\em E.J. Combinatorics}, 12:283--291, 1991.

\bibitem{CLM15}
Lucia Caporasoa, Yoav Lenb, and Margarida Meloa.
\newblock Algebraic and combinatorial rank of divisors on finite graphs.
\newblock {\em Journal de Mathématiques Pures et Appliquées},
  104(2):227--257, 2015.

\bibitem{CL14}
Robert Cori and Yvan~Le Borgne.
\newblock The riemann-roch theorem for graphs and the rank in complete graphs.
\newblock {\em http://arxiv.org/abs/1308.5325}, 2014.

\bibitem{Das10}
Kalyani Das.
\newblock An optimal algorithm to find maximum independent set and maximum
  2-independent set on cactus graphs.
\newblock {\em AMO - Advanced Modeling and Optimization}, 12(2):239--248, 2010.

\bibitem{GLMMP01}
E.~Goles, M.~Latapy, C.~Magnien, M.~Morvan, and H.~D. Phan.
\newblock Sandpile models and lattices: a comprehensive survey.
\newblock {\em Theoret. Comput. Sci.}, 322(2):383--407, 2004.

\bibitem{FG53}
Frank Harary and George~E Uhlenbeck.
\newblock On the number of husimi trees, i.
\newblock {\em Proceedings of the National Academy of Sciences},
  39(4):315–322, 1953.

\bibitem{Hobbs79}
Arthur~M Hobbs.
\newblock Hamiltonian squares of cacti.
\newblock {\em Journal of Combinatorial Theory, Series B}, 26(1):50--65, 1979.

\bibitem{KT15}
Viktor Kiss and Lilla Tothmeresz.
\newblock Chip-firing games on eulerian digraphs and np-hardness of computing
  the rank of a divisor on a graph.
\newblock {\em arXiv:1407.6958v3 [cs.CC]}, 2015.

\bibitem{Tetsuo91}
Tetsuo Nishi.
\newblock On the number of solutions of a class of nonlinear resistive circuit.
\newblock {\em Proceedings of the IEEE International Symposium on Circuits and
  Systems. 1991}.

\bibitem{PDESMSH}
Benedict Paten, Mark Diekhans, Dent Earl, John St.~John, Jian Ma, Bernard Suh,
  and David Haussler.
\newblock Research in computational molecular biology.
\newblock {\em Lecture Notes in Computer Science, Lecture Notes in Computer
  Science}, 6044:766–769, 2010.

\bibitem{PP15}
Kevin Perrot and Trung Van~Pham.
\newblock Feedback arc set problem and np-hardness of minimum recurrent
  configuration problem of chip-firing game on directed graphs.
\newblock {\em Annals of Combinatorics}, 19(2):373--396, 2015.

\bibitem{PP13}
Trung~Van Pham and Thi Ha~Duong Phan.
\newblock Lattices generated by chip firing game models: criteria and
  recognition algorithms.
\newblock {\em European Journal of Combinatorics}, 34(5):812--832, 2013.

\bibitem{Yu83}
K.~E. Yu.
\newblock Representation of temporal knowledge.
\newblock {\em Proc. 8th International Joint Conference on Artificial
  Intelligence}, 1983.

\bibitem{ZZ03}
Blaz Zmazek and Janez Zerovnik.
\newblock Computing the weighted wiener and szeged number on weighted cactus
  graphs in linear time.
\newblock {\em Croatica Chemica Acta}, 76(2):137 -- 143, 2003.

\end{thebibliography}
\bibliographystyle{plain}

\end{document}